\documentclass[]{interact}

\usepackage{epstopdf}
\usepackage[caption=false]{subfig}

\usepackage[numbers,sort&compress]{natbib}
\bibpunct[, ]{[}{]}{,}{n}{,}{,}
\makeatletter
\def\NAT@def@citea{\def\@citea{\NAT@separator}}
\makeatother

\theoremstyle{plain}
\newtheorem{theorem}{Theorem}[section]

\newtheorem{proposition}[theorem]{Proposition}

\theoremstyle{definition}

\theoremstyle{remark}
\newtheorem{remark}{Remark}

\usepackage{xcolor}
\definecolor{webgreen}{rgb}{0,.5,0}
\definecolor{webbrown}{rgb}{.6,0,0}
\definecolor{RoyalBlue}{cmyk}{1, 0.50, 0, 0}
\usepackage[colorlinks=true, breaklinks=true, urlcolor=webbrown, linkcolor=RoyalBlue, citecolor=webgreen,backref=page]{hyperref}
\usepackage{amsmath,amssymb,amsthm,upgreek}
\usepackage{enumerate}
\usepackage{verbatim}

\newcommand{\R}{{\mathbb R}}
\newcommand{\N}{{\mathbb N}}
\newcommand{\C}{{\mathbb C}}

\newcommand{\Z}{{\mathbb Z}}

\newcommand{\ic}{{\mathrm i}}
\newcommand{\dd}{{\mathrm d}}

\newcommand{\im}{\mathrm{Im}}
\newcommand{\re}{\mathrm{Re}}

\newcommand{\RS}{\boldsymbol{\mathfrak{R}}}

\newcommand{\z}	{{\boldsymbol z}}

\newcommand{\s}	{{\boldsymbol s}}

\newcommand{\qandq}{\quad \text{and} \quad}
\newcommand{\qasq}{\quad \text{as} \quad }
\newcommand{\qforq}{\quad \text{for} \quad }

\newcommand{\rhy}   {\textnormal{RHP}-${\boldsymbol Y}$}
\newcommand{\rht}   {\textnormal{RHP}-${\boldsymbol T}$}

\newcommand{\rhx}   {\textnormal{RHP}-${\boldsymbol X}$}
\newcommand{\rhn}   {\textnormal{RHP}-${\boldsymbol N}$}

\newcommand{\rhpc}   {\textnormal{RHP}-${\boldsymbol P}_{c}$}

\newcommand{\be}{\begin{equation}}
\newcommand{\ee}{\end{equation}}
\begin{document}

\articletype{}

\title{Strong Asymptotics of Jacobi-Type Kissing Polynomials}

\author{
\name{A.~B. Barhoumi\textsuperscript{a}\thanks{CONTACT A.~B. Barhoumi. Email: barhoumi@umich.edu}}
\affil{\textsuperscript{a}Department of Mathematical Sciences\\
	Indiana University-Purdue University Indianapolis\\
	402~North Blackford Street, Indianapolis, IN 46202}
}

\maketitle

\begin{abstract}
We investigate asymptotic behavior of polynomials \( p^{\omega}_n(z) \) satisfying varying non-Hermitian orthogonality relations
\[
\int_{-1}^{1} x^kp^{\omega}_n(x)h(x) e^{\ic \omega x}\dd x =0, \quad k\in\{0,\ldots,n-1\},
\]
where $h(x) = h^*(x) (1 - x)^{\alpha} (1 + x)^{\beta}, \ \omega = \lambda n, \ \lambda \geq 0 $ and $h(x)$ is holomorphic and non-vanishing in a certain neighborhood in the plane. These polynomials are an extension of so-called kissing polynomials ($\alpha = \beta = 0$) introduced in \cite{asheim12gaussian} in connection with complex Gaussian quadrature rules with uniform good properties in $\omega$. The analysis carried out here is an extension of what was done in \cite{celsus19, deano14}, and depends heavily on those works. 
\end{abstract}

\begin{keywords}
Non-Hermitian orthogonality, varying orthogonality, Riemann-Hilbert analysis
\end{keywords}

\section{Introduction}

The purpose of this note is to extend the work done in connection with complex quadrature rules for oscillatory integrals 
\[
\int_{-1}^{1} f(x) e^{\ic \omega x} \dd x.
\]
Evaluation of such integrals via the standard Gaussian quadratures can become extremely expensive numerically for large values of $\omega$, motivating the development of new quadrature rules. It was shown in \cite{asheim12gaussian} that using the zeros of polynomials $p_n^{\omega}$ which satisfy 
\be
\label{ortho-cond}
\int_{-1}^{1} x^kp^{\omega}_n(x) h(x) e^{\ic \omega x}\dd x =0, \quad k\in\{0,\ldots,n-1\},
\ee
where $h(x) = 1$ identically yields a quadrature rule with `good` properties that naturally reduces to the usual quadrature rule when $\omega \to 0$. For more on this and different computational methods, see the monograph by Dea\~no, Huybrechs, and Iserles \cite{DHI}. In this note, we will be interested in the asymptotic analysis of the polynomials $p_n^{\omega}$ arising in the slightly more general situation where 
\be
\label{h}
h(x) = h^*(x)(1 -x)^{\alpha} (1 + x)^{\beta},  \ \alpha, \beta > -1
\ee
and $h^*(z)$ is holomorphic in a certain region of the plane. 


\subsection*{Overview of the Paper}
Three regimes, separated by the geometry of the zero-attracting curve associated with $p_n^{\omega}$ (denoted $\gamma_{\lambda}$), are considered in this work. The main tool for the analysis carried out in all three regimes is the Riemann-Hilbert problem (RHP) for orthogonal polynomials and the Deift-Zhou nonlinear steepest descent method, where the initial RHP is transformed to a normalized RHP with the help of the so-called $g$-function, and a global parametrix and a set of local parametrices are constructed. The details of these constructions differ from one regime to the other, and depend on the geometry of $\gamma_{\lambda}$. With this in mind, Section \ref{geo} serves as a quick reminder of results pertaining to the zero-attracting curve associated with $p_n^{\omega}$ for all possible values of $\lambda \in [0, \infty)$. In Section \ref{statement} asymptotic formulas for $p_n^{\omega}(z)$ are stated for the subcritical, critical, and supercritical regimes for $z \in \C \setminus \gamma_{\lambda}$. Similar formulas can be obtained for $z \in \gamma_{\lambda}$, but such calculation is omitted for brevity. Proof of the formula for the supercritical case is provided in Section \ref{two-cut-pf}, and sketches of the proofs for the subcritical and critical regimes are provided in Sections \ref{one-cut-pf}, \ref{crit-pf}, respectively.\\

This work should be viewed as an extension of the work in \cite{deano14, celsus19}. Some of the main differences include analyzing polynomials $p_n^{\omega}$ in the critical case, allowing for more general weights, including ones with an algebraic singularities at the end-points $z = 1, z = -1$, and using a different construction of the global parametrix while analyzing the supercritical regime than the one studied in \cite{celsus19} (compare leading term in \eqref{2-cut-thm-formula} below with the one obtained in \cite[Theorem 2.4]{celsus19}). 
\section{Geometry}
\label{geo} 
Since the weight of orthogonality is complex-valued, it is known that the zeros of $p_n^{\omega}$ may not accumulate onto the interval $[-1, 1]$. It turns out that $[-1, 1]$ is the zero-attracting curve in the case where the value $\omega$ is fixed (see the appendix of \cite{deano14}). When $\omega$ is allowed to vary with $n$ as $\omega = \lambda n, \ \lambda \geq 0$, the situation becomes more interesting as we enter the world of varying orthogonality. The work of Gonchar and Rakhmanov \cite{GoncharRakhmanov89} suggested that one ought to consider a curve $\gamma_{\lambda}$ to which $[-1, 1]$ is deformable and satisfies the S-property:
\[
\dfrac{\partial\left(  U^{\mu_{\lambda}} + \re(V)\right)}{\partial {\bf n}^+} (z) = \dfrac{\partial \left(  U^{\mu_{\lambda}} + \re(V)\right)}{\partial {\bf n}^-} (z) \quad \forall z \in \gamma_{\lambda},
\]
where $U^{\mu_{\lambda_{\lambda}}}(z) := -\int \log|z - s| \dd \mu_{\lambda}(s)$ and $\mu_{\lambda}$ is the equilibrium measure on $\gamma_{\lambda}$ in the external field $\re(V)$ (in our setting, $V(z) = -\ic \lambda z$). They further show that such curves are formed by the trajectories of a quadratic differential $-Q_{\lambda}(z) (\dd z)^2$ where $Q_{\lambda}$ is given by 
\be
\label{q-kissing}
Q_{\lambda}(z) = \left(  \int \dfrac{\dd \mu_{\lambda}(s)}{s - z} + \dfrac{V'(z)}{2} \right)^2 = \left(  \int \dfrac{\dd \mu_{\lambda}(s)}{s - z} - \dfrac{\lambda \ic}{2} \right)^2.
\ee
To obtain a formula for $Q_{\lambda}$, it is common to assume something about the support of $\mu_{\lambda}$ to be proven later on. This was done by Dea\~no, who showed the following: define  
\begin{equation}
\label{w-1}
\varphi(z) := z + w(z), \quad w(z) = (z^2 - 1)^{1/2}, \ z\in \C \setminus \gamma_{\lambda}, \quad w(z) = z + \mathcal{O}(z) \qasq z \to \infty 
\end{equation} 
and let $\lambda_{cr}$ be the unique solution of 
\be
\label{lambda-crit}
2 \log \left( \dfrac{2 + \sqrt{\lambda_{cr}^2 + 4}}{\lambda_{cr}}  \right) - \sqrt{\lambda_{cr}^2 + 4} = 0 \quad (\lambda_{cr} \approx 1.325...).
\ee
The following theorem appeared in \cite{deano14}:
\begin{theorem}
	\label{kissing-geo-one-cut}
	Let $V(z) = -\ic \lambda z$ and $\lambda \in [0, \lambda_{cr})$. Then,
	\begin{enumerate}
		\item there exists a smooth curve $\gamma_{\lambda}$ connecting $z = 1$ and $z = -1$ that is a part of the level set 
		\(
		\re(\phi(z)) = 0 
		\)
		where 
		\be
		\label{phi-kiss}
		\phi(z) =  2\log \varphi(z) + \ic \lambda w(z). 
		\ee
		\item The measure 
		\(
		\dd \mu_{\lambda} (z) = -\dfrac{1}{2\pi \ic} \dfrac{2 + \ic \lambda z}{w(z)} \ \dd z
		\) is the equilibrium measure on $\gamma_{\lambda}$ in the external field $\re(V(z))$. 
		\item $\gamma_{\lambda}$ has the S-property in the field $\re(V(z))$. 
	\end{enumerate}
\end{theorem}
\begin{remark}
In fact, Dea\~no's proof shows that for $\lambda = \lambda_{cr}$, $\gamma_{\lambda}$ is a union of two smooth curves that meet at $2\ic/\lambda_{cr}$.
\end{remark}
\begin{remark}
Observe that with this theorem in mind, one can calculate the function $Q_{\lambda}(z)$ via Privalov's lemma and \eqref{q-kissing} and find
\be
\label{q-form-kissing}
Q_{\lambda}(z) = \dfrac{1}{4}\dfrac{(2 + \ic \lambda z)^2}{z^2 - 1}.
\ee
\end{remark}
As for the supercritical case $\lambda \in (\lambda_{cr}, \infty)$, Celsus and Silva showed in \cite{celsus19} that (most of) the zeros of $p_{n}^{\omega}$ accumulate on two disconnected arcs, $ \gamma_{1}, \gamma_{2}$ (which depend on $\lambda$) that appear as trajectories of the quadratic differential $-Q_{\lambda}(z; x_*) (\dd z)^2$ where 
\be
\label{q-form-2-cut}
Q_{\lambda}(z; x) := -\dfrac{\lambda^2}{4} \dfrac{(z - z_{\lambda}(x))(z + \overline{z_{\lambda}(x)})}{z^2 - 1}, \qandq z_{\lambda}(x) = x + \dfrac{2 \ic}{\lambda},
\ee
and $x_*$ is some special value. More precisely, combining their work with Gonchar and Rakhmanov's, we deduce that the counting measure associated with $p_n^{\omega}$ weakly converges to $\mu_{\lambda}$. The density of $\mu_{\lambda}$ is given in the following theorem, due to Celsus and Silva \cite{celsus19}:
\begin{theorem}
	\label{2-cut-geo}
	Let $\lambda > \lambda_{cr}$ and define $Q_{\lambda}(z):= Q_{\lambda}(z, x_*)$,  where $x_*(\lambda) \in (0, 1)$ is the unique value for which 
	\(
	\re \left( \int_{z_{\lambda}(x_*)}^1 Q_{\lambda}(s) \dd s  \right) = 0
	\)
	and $\lim_{\lambda \to \infty}x_* (\lambda) = 1$. Then, there exist analytic arcs $\gamma_{1}, \gamma_{2}$ such that $\gamma_{1}$ is an arc connecting $-1$ to $-\overline{z_{\lambda}(x_*)}$ that lies in the left half-plane, $\gamma_{2}$ is the reflection of $\gamma_{1}$ across the imaginary axis, and they satisfy  
		\be
		\re \left(\int_{-1}^z Q_{\lambda}^{1/2}(s) \dd s \right) = 0 \quad \forall z \in \gamma_{1} \qandq \re \left(\int_{z_{\lambda}(x_*)}^z Q_{\lambda}^{1/2}(s) \dd s \right) = 0 \quad \forall z \in \gamma_{2}.
		\ee

	The equilibrium measure, $\mu_{\lambda}$ in the external field $\re(V)$ has the following density 
	\be
	\label{eq-meas-2-cut}
	\dd \mu_{\lambda}(s) = -\dfrac{1}{\pi \ic} Q_{\lambda}^{1/2}(s) \ \dd s, \quad s \in \gamma_1 \cup \gamma_2,
	\ee
	where we take the branch of $Q_{\lambda}^{1/2}$ holomorphic in $\C \setminus (\gamma_1 \cup \gamma_2)$ and behaves as $Q_{\lambda}^{1/2}(z) = \lambda \ic/2  + \mathcal O\left( z^{-1} \right)$ as $z \to \infty.$
\end{theorem}
\begin{figure}[ht!]
	\centering
	\includegraphics[width = 0.5 \textwidth]{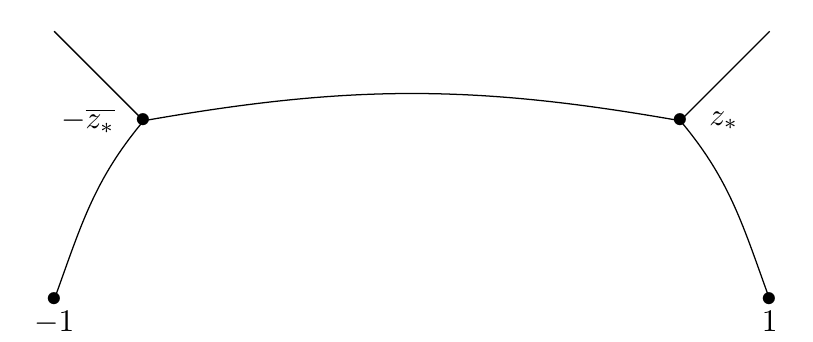}
	\caption{Schematic representation of critical graph of $-Q_{\lambda}(z) \ (\dd z)^2$ in the supercritical regime near $z = -1, z = 1$, with $z_* := z_{\lambda}(x_*)$. See \cite{celsus19} for a proof.}
	\label{fig:2-cut-kiss-graph}
\end{figure}

\section{Statement of Results}
\label{statement}

\subsection{Asymptotics: One-cut Case}

Let $\lambda_{cr}$ be as in \eqref{lambda-crit}. In the non-critical case ($\lambda < \lambda_{cr}$), the situation was described completely for $h(x) = 1$ identically in \cite{deano14}. To extend this result to $h(x)$ as in \eqref{h}, we need the following Szeg\H{o} function
\begin{equation}
\label{szego}
S_h(z) := \exp \left\{ \dfrac{w(z)}{2 \pi \ic}  \int_{\gamma_{\lambda}} \dfrac{\log[(w_+ h)(x) ]}{z - x} \dfrac{\dd x}{w_+(x)} \right\}, \quad z \in \C \setminus \gamma_{\lambda},
\end{equation}
where $w$ is as in \eqref{w-1} and $h^*(z)$ is holomorphic in a neighborhood containing the compact set delimited by $\gamma_{\lambda} \cup [-1, 1]$. Properties of $S_h$ will be discussed in Section \ref{one-cut-pf}.
\begin{theorem}[{\bf Subcritical Case $\lambda < \lambda_{cr}$}]
	\label{one-cut-thm}
	Let $0 \leq \lambda < \lambda_{cr}$ and $h(z)$ be as above. Then for $n$ large enough, polynomials $p_n^{\omega}$ have degree exactly $n$ and locally uniformly for $z \in \C \setminus \gamma_{\lambda}$ 
	\begin{equation}
	\label{oc-p-asymp}
	p_n^{\omega}(z) = \left(\dfrac{\varphi(z)}{2} \right)^n \exp \left(-\dfrac{\ic n \lambda}{2 \varphi(z)}\right) \left(\dfrac{S_h(\infty)}{S_h(z)} + \mathcal{O}\left(n^{-1}\right)\right) \qasq n \to \infty.
	\end{equation}
\end{theorem}
When $\lambda = \lambda_{cr}$, the geometry of $\gamma_{\lambda}$ changes. More precisely, $\gamma_{\lambda}$ is no longer an analytic arc, but rather a union of two analytic arcs, see \cite{deano14}. However, by slightly changing the analysis, we may still write an asymptotic formula for $p_n^{\omega}$. 
\begin{theorem}[{\bf Critical Case $\lambda = \lambda_{cr}$}]
	\label{crit-thm}
	Let $\lambda = \lambda_{cr}$ and $h(z)$ be as above. Then for $n$ large enough, polynomials $p_n^{\omega}$ have degree exactly $n$ and locally uniformly for $z \in \C \setminus \gamma_{\lambda}$ 
	\begin{equation}
	\label{crit-p-asymp}
	p_n^{\omega}(z) = \left(\dfrac{\varphi(z)}{2} \right)^n \exp \left(-\dfrac{\ic n \lambda}{2 \varphi(z)}\right) \left(\dfrac{S_h(\infty)}{S_h(z)} + \mathcal{O}\left(n^{-1/2}\right)\right) \qasq n \to \infty.
	\end{equation}
\end{theorem}
We sketch a proof of Theorems \ref{one-cut-thm}, \ref{crit-thm} in Sections \ref{one-cut-pf}, \ref{crit-pf}, respectively.

\subsection{Asymptotics: Two-cut Case}
To present the results when $\lambda > \lambda_{cr}$, we construct the main term of the asymptotics using the approach of \cite{ApY15} relying on Theta functions, instead of the meromorphic differential approach taken in \cite{celsus19}. We introduce those here. Let $z_* = z_{\lambda}(x_*)$ (see Theorem \ref{2-cut-geo}) and
\begin{equation}
\label{gamma-kissing}
\gamma(z):=\left(\frac{z+\overline{z_*}}{z  -z_*}\frac{z-1}{z+1}\right)^{1/4}, \quad z\in\overline\C\setminus (\gamma_1 \cup \gamma_2),
\end{equation}
where $\gamma(z)$ is holomorphic off $\gamma_1 \cup \gamma_2$ and the branch is chosen so that $\gamma(\infty)=1$. Further, set
\begin{equation}
\label{ab-kissing}
A(z) = \frac{\gamma(z)+\gamma^{-1}(z)}2 \qandq B(z) := \frac{\gamma(z)-\gamma^{-1}(z)}{-2\ic}.
\end{equation}
The functions \( A(z) \) and \( B(z) \) are holomorphic in \( \overline\C\setminus (\gamma_1 \cup \gamma_2) \), $A(\infty)=1, \quad B(\infty)=0$, and 
\begin{equation}
\label{nt3}
A_\pm(s) = \pm B_\mp(s), ~~ s \in (\gamma_1 \cup \gamma_2)\setminus \{ \pm 1, z_*, -\overline{z_*} \}.
\end{equation}
\subsubsection{Riemann Surface}
\label{riemann-surface}
Let $\RS$ be the Riemann surface associated with the algebraic equation $y^2 = Q_{\lambda}(z)$, with $Q_{\lambda}$ as in Theorem \ref{2-cut-geo}. This surface is realized as two copies of $\C$ cut along $\gamma_{1, 2}$ and glued together in such a way that the right side of $\gamma_{i}$ on $\RS^{(0)}$, the first sheet, is connected with the left side of the same arc on the second sheet, $\RS^{(1)}$. Furthermore, $\pi: \RS \to \overline{\C}$ be the natural projection. We will denote points on the surface with boldface symbols $\boldsymbol z, \boldsymbol t, \boldsymbol s$ and their projections by regular script $z, s, t$ and \( F^{(i)}(z) \), \( i\in\{0,1\} \), stands for the pull-back under \( \pi(\z) \) of a function \( F(\z) \) from \( \RS^{(i)} \) into \( \overline\C\setminus (\gamma_1 \cup \gamma_2) \). Note that for a fixed $z \in \C \setminus (\gamma_{1} \cup \gamma_2)$, the set $\pi^{-1}(z)$ contains exactly two elements, one on each sheet, and we denote by $z^{(k)}$ the unique point satisfying $z^{(k)} \in \pi^{-1}(z) \cap \RS^{(k)}$.

Denote by \( \boldsymbol\alpha \) a cycle on \( \RS \) that passes through \( \pi^{-1}(-\overline{z_*}) \) and \( \pi^{-1}(z_*) \) and whose natural projection is the arc $\hat{\gamma}$ that smoothly meets $\gamma_1, \gamma_2$ at $z_*, -\overline{z_*}$, belongs to the region delimited by infinite trajectories in Figure \ref{fig:2-cut-kiss-graph}, and agrees with the orthogonal trajectory of $-Q(z) (\dd z)^2$ in a small neighborhood of $z_*, -\overline{z_*}$. We assume that \( \pi(\boldsymbol\alpha)\cap (\gamma_1 \cup \gamma_2)=\{z_*,-\overline{z_*}\} \) and orient \( \boldsymbol\alpha \) towards \( \boldsymbol -\overline{z_*} \) within \( \RS^{(0)} \). Similarly, we define \( \boldsymbol\beta \) to be a cycle on \( \RS \) that passes through \( \pi^{-1}(-1) \) and \( \pi^{-1}(-\overline{z_*}) \) and whose natural projection is $\gamma_1$. We orient \( \boldsymbol\beta \) so that \( \boldsymbol\alpha,\boldsymbol\beta \) form the right pair at \( \pi^{-1}(-\overline{z_*}) \).

Since this is a surface of genus 1, the linear space of holomorphic differentials is of dimension 1, and is generated by (we slightly abuse the notation $w$ here)
\begin{align}
\label{hol-diff-kiss}
& \mathcal{H}(\z) := \left(\oint_{\boldsymbol \alpha} \dfrac{\dd t}{w(\boldsymbol t)} \right)^{-1} \dfrac{\dd z}{w(\boldsymbol z)},
\end{align}
where 
\begin{align}
w(z^{(k)}) &= (-1)^k \left[(z^2 - 1)(z - z_*)(z + \overline{z_*}) \right]^{1/2}(z), \ z \in \C \setminus (\gamma_{1} \cup \gamma_2), \label{w-2} \\
w(z^{(k)}) &= (-1)^kz^2 + \mathcal{O}(z) \qasq z \to \infty.
\end{align}
$\mathcal{H}$ is normalized so that $\oint_{\boldsymbol \alpha} \mathcal H = 1$, and under this normalization, Riemann showed that 
\be
\label{B-kiss}
\im(\mathsf B) > 0, \quad \text{where } \quad \mathsf B := \oint_{\boldsymbol \beta} \mathcal H.
\ee
Given this normalized differential, we can define the Abel Map $\mathcal A (\z):= \int_{1}^{\z} \mathcal H$ where the path of integration is chosen to lie in $\RS_{\boldsymbol \alpha, \boldsymbol \beta}:=\RS \setminus \{ \boldsymbol \alpha, \boldsymbol \beta \}$. This function is holomorphic on $\RS_{\boldsymbol \alpha, \boldsymbol \beta}$ that satisfies 
\be
\label{abel-jumps}
(\mathcal A_+ - \mathcal A_-)(\z) = \left\{ \begin{array}{ll}
	1, & \z \in \boldsymbol \beta \setminus \pi^{-1}(-1),\\ 
	-\mathsf B, & \z \in \boldsymbol \alpha \setminus \pi^{-1}(-1).
\end{array}\right.
\ee

\subsubsection{Szeg\H{o} Function}
Let 
\be
\label{szego-kiss}
\tilde{S}_h(z^{(k)}) := \exp \left\{ \dfrac{1}{4 \pi \ic} \oint_{\pi^{-1}(\gamma_1 \cup \gamma_2)} \log (h) \Omega_{z^{(k)}, z^{(1-k)}} \right\} \qforq k = 0, 1,
\ee
where $w$ is as in \eqref{w-2} and $\Omega_{z^{(k)}, z^{(1-k)}}$ is the meromorphic differential on $\RS$ with simple pole at $z^{(k)}, z^{(1 - k)}$ with residues $1, -1$, respectively and $\int_{\boldsymbol \alpha} \Omega_{z^{(k)}, z^{(1 - k)}} = 0$.
\begin{proposition}
	\label{prop:szego-kiss}
	Let $\tilde{S}_h$ be as above and $h(z) = h^*(z) (1 - z)^{\alpha}(1 + z)^{\beta}$ where $h^*(z)$ is holomorphic, non-vanishing in a neighborhood of $\gamma_1 \cup \gamma_{2} \cup \hat \gamma$ and $h(z)$ is holomorphic in a neighborhood of each point of $(\gamma_1 \cup \gamma_2) \setminus \{\pm 1, z_*, -\overline{z_*} \}$. Furthermore, define 
	\be
	\label{kiss-c-h}
	c_h = c_h(\lambda):= \dfrac{1}{2\pi \ic} \oint_{\pi^{-1}(\gamma_1 \cup \gamma_2)} \log(h) \mathcal{H}.
	\ee
	Then $\tilde{S}_h$ is holomorphic and non-vanishing on $\RS \setminus \{\boldsymbol \alpha, \pi^{-1}(\gamma_1 \cup \gamma_2)\}$ and satisfies the relation $\tilde{S}_h(z^{(k)}) \cdot \tilde{S}_h(z^{(1 - k)}) = 1$ identically. Furthermore, $\tilde{S}_{h}$ possesses continuous traces on $\boldsymbol \alpha \cup \pi^{-1}(\gamma_1 \cup \gamma_2) \setminus \{\pi^{-1}(\pm 1) \pi^{-1}(z_*), \pi^{-1}(-\overline{z_*})\}$ that satisfy 
	\be
	\label{szego-kiss-jump}
	\tilde{S}_{h, +}(\s) = \tilde{S}_{h, -}(\s) \left\{ \begin{array}{ll}
		e^{2\pi \ic c_h}, & s \in \boldsymbol \alpha \setminus \{ z_*, -\overline{z_*}\},\\
		1/h(s), & s \in \pi^{-1}(\gamma_{1} \cup \gamma_2)\setminus \{ \pi^{-1}(\pm 1) \}.
	\end{array}\right.
	\ee
	Furthermore, we have 
	\(
	\tilde{S}_h(z^{(0)}) = |z - e|^{-\alpha_e/ 2}, \quad e \in \{ \pm 1, z_*, -\overline{z_*} \}  ,
	\)
	where $\alpha_e = 0$ for $e = z_*, -\overline{z_*}$, $\alpha_e = \alpha$ when $e = 1$ and $\alpha_e = \beta$ when $e = -1$. 
\end{proposition}
For a proof of this, see \cite[Section 6.1]{ApY15}. 
\subsubsection{Theta Function}
\label{theta-kiss}
Let $\theta(z)$ be the function defined by the sum
\(
\theta(u) = \sum_{k \in \Z} \exp \left\{ \pi \ic \mathsf B k^2 + 2\pi \ic u k  \right\}.
\)
For convenience, we remind the reader of its properties here. This function is holomorphic in $\C$ and satisfies the quasi-periodicity relations
\be
\theta(u+j+\mathsf B m) = \exp\left\{-\pi\ic \mathsf B m^2-2\pi\ic um\right\}\theta(u), \qquad j,m\in\Z.
\ee
It is also known that $\theta(u)$ vanishes only at the points of the lattice \( \frac{\mathsf B+1}2 + m + n \mathsf B, \ m, n \in \Z \). Furthermore, let $\tilde{\mathcal A}$ denote the continuation of $\mathcal A$ onto $\boldsymbol \alpha, \boldsymbol \beta$ by $\mathcal{A}_+$ and define $\z_{n, k}$ by the equation
\be
\label{jip}
\tilde{\mathcal A}(\z_{n, k}) = \tilde{\mathcal A}\left(p^{(k)} \right) + c_h + n\left( \dfrac{1}{2} + \mathsf B \tau \right) + j_{n, k} + m_{n,k} \mathsf B, \ j_{n, k}, m_{n,k} \in \Z,
\ee
where $p = \ic \im(z_*)/(1 - \re(z_*)))$ and 
\be
\label{kissing-periods}
\tau  := -\dfrac{1}{\pi \ic}\int_{\hat{\gamma}} Q_{\lambda}^{1/2}(s) \dd s.
\ee 
Since $\RS$ is of genus one, $\mathcal A$ is bijective and equation \eqref{jip} defines $\z_{n, k}$ uniquely. In fact, by considering the branch choices in the definition of $A, B$, the following holds.
\begin{proposition}
	\label{prop:Nt}
	Let \( \tau \) be given by \eqref{kissing-periods}, \( \z_{n,k}=\z_{n,k}(\lambda) \) as in \eqref{jip}, and $p$ as above. Then for any subsequence \( \N_* \) the point \( \infty^{(0)} \) is a topological limit point of \( \{\z_{n,1}\}_{n\in \N_*} \) if and only if \( \infty^{(1)} \) is a topological limit point of \( \{\z_{n,0}\}_{n\in \N_*} \). 
\end{proposition}

\begin{proof}
	
	It follows from \eqref{nt3} and choice of the branch of $(\cdot)^{1/4}$ that $\gamma(p) = 1$ and
	\begin{equation}
	\label{nt4}
	\left\{
	\begin{array}{rl}
	(B/A)(z), & \z\in \RS^{(0)}, \medskip \\
	-(A/B)(z), & \z\in \RS^{(1)},
	\end{array}
	\right.
	\end{equation}
	is a rational function on \( \RS \) with two simple zeros $\infty^{(0)}$ and $p^{(0)}$ and two simple poles $\infty^{(1)}$ and $p^{(1)}$ (if it happens that \( p\in (\gamma_1 \cup \gamma_2)\setminus \{\pm 1, z_*, -\overline{z_*}\} \), then we choose \( p^{(0)}\in\RS \) precisely in such a way that it is a zero of \eqref{nt4} and \( p^{(1)} \) so it is a pole of \eqref{nt4}; it is, of course, still true that these points are distinct and \( \pi\big(p^{(k)}\big) = p \)). Therefore, Abel's theorem yields that
	\begin{equation}
	\label{nt5}
	\int_{p^{(0)}}^{\infty^{(1)}}\mathcal H  = \int_{p^{(1)}}^{\infty^{(0)}}\mathcal H \quad \text{ modulo } \Z + \mathsf B \Z, 
	\end{equation}
	while the relations \eqref{jip}, in particular, imply that
	\begin{equation}
	\label{nt6}
	\int_{p^{(0)}}^{\z_{n,0}}\mathcal H  =\int_{p^{(1)}}^{\z_{n,1}}\mathcal H \quad \text{ modulo } \Z + \mathsf B \Z.
	\end{equation}
	Let $\z_k$ be a topological limit of a subsequence $\{\z_{n_i,k}\}$. Holomorphy of the differential \( \mathcal H \) implies that
	\[
	\int_{p^{(k)}}^{\z_{n_i,k}} \mathcal H = \int_{p^{(k)}}^{\z_k} \mathcal H + \int_{\z_k}^{\z_{n_i,k}} \mathcal H \to \int_{p^{(k)}}^{\z_k} \mathcal H
	\]
	as \( i\to\infty \), where the integral from \( \z_k \) to \( \z_{n_i,k} \) is taken along the path that projects into a segment joining \( z_k \) and \( z_{n_i,k} \). The desired claim now follows from \eqref{nt5}, \eqref{nt6}, and the unique solvability of the Jacobi inversion problem on \( \RS \).
\end{proof}

Now, we define 
\begin{equation}
\label{ab5}
\Theta_{n,k}(\z) = \exp\left\{-2\pi\ic\big(m_{n,k}+\tau n\big)\mathcal A(\z)\right\} \frac{\theta\left(\mathcal A(\z) - \tilde{\mathcal A}(\z_{n,k}) - \frac{\mathsf B+1}2\right)}{\theta\left(\mathcal A(\z) - \tilde{\mathcal A}\big(p^{(k)}\big) - \frac{\mathsf B+1}2\right)}.
\end{equation}
The functions $\Theta_{n,k}(\z)$ are meromorphic on \( \RS_{\boldsymbol\alpha,\boldsymbol\beta} \) with exactly one pole, which is simple and located at $p^{(k)}$, and exactly one zero, which is also simple and located at \( \z_{n,k} \) (observe that the functions $\Theta_{n,k}(\z)$ can be analytically continued as multiplicatively multivalued functions on the whole surface \( \RS \); thus, we can talk about simplicity of a pole or zero regardless whether it belongs to the cycles of a homology basis or not). Moreover, according to \eqref{abel-jumps}, \eqref{jip}, and periodicity properties of $\theta$, they possess continuous traces on $\boldsymbol\alpha,\boldsymbol\beta$ away from \( \pi^{-1}(-1) \) that satisfy
\begin{equation}
\label{ab6}
\Theta_{n,k+}(\s) = \Theta_{n,k-}(\s)\left\{
\begin{array}{rl}
\exp\big\{-\pi\ic(n + 2c_h)\big\}, & \boldsymbol s\in\boldsymbol\alpha\setminus\{\pi^{-1}(-1)\}, \medskip \\
\exp\big\{-2\pi\ic\tau n\big\}, & \boldsymbol s\in\boldsymbol\beta\setminus\{\pi^{-1}(-1)\}.
\end{array}
\right.
\end{equation}

\subsubsection{Subsequences \texorpdfstring{\( \N(\lambda,\varepsilon) \)}{ }}
It will be important for our analysis (see section \ref{global}) that \( \Theta_{n, 1}(\z;\lambda) \), defined in \eqref{ab5}, does not vanish near \( \infty^{(0)} \). Hence, we will consider subsequences \( \N(\varepsilon)=\N(\lambda,\varepsilon) \) are defined as 
\[
\N(\varepsilon) := \left\{ n\in\N:~~\z_{n,1} \not\in \RS^{(0)} \cap \pi^{-1}\big(\big\{|z|\geq 1/\varepsilon\big\}\big) \right\}.
\]
Then there exists a constant \( c(\lambda,\varepsilon)>0 \) such that $|\Theta^{(1)}_{n, 1}(\infty;t)| \geq c(\lambda,\varepsilon)$ for $n \in \N(\lambda,\varepsilon).$ Note that $\N(\lambda, \epsilon)$ contains $n$ or $n - 1$ for all $n \geq 1$. To prove this, suppose to the contrary that for any $\epsilon>0$, there exists $n_{\epsilon}$ such that $n_{\epsilon}, \ n_{\epsilon} - 1 \not\in \N(\lambda, \epsilon)$. By the very definition of \( \N(\lambda, \epsilon) \), it then holds that $\z_{n_{\varepsilon}-1 ,1}, \ \z_{n_{\varepsilon}, 1} \to \infty^{(0)}$ as \( \varepsilon\to 0 \). This implies $1/2 + \mathsf B \tau = m + n \mathsf{B}$ for some $m, n \in \Z$, which is false. We are ready to state the asymptotic formula for $p_n^{\omega}(z)$.

\begin{theorem}[{\bf Supercritical Case $(\lambda > \lambda_{cr})$}]
\label{2-cut-thm}
Let $\lambda > \lambda_{cr}$,  $V(z) = -\ic \lambda z$, $h(z)$ as in Proposition \ref{prop:szego-kiss}, and $\phi_1(z) = \int_1^z Q^{1/2}_{\lambda}(s) \dd s$. Then, there exists a constant $\ell^*$ (defined in \eqref{phi-kiss-2}) so that 
\be
\label{2-cut-thm-formula}
p_n^{\omega}(z) = e^{n(V(z) - \ell^* + \phi_1(z))} \left( \left(A \Theta_{n, 1}^{(0)}  \tilde{S}_h^{(0)}\right)(z) + \mathcal{O} \left( n^{-1}\right) \right)	\qforq n\to \infty, \ n \in \N(\lambda, \epsilon)
\ee
locally uniformly for $z \in \C \setminus \gamma_{\lambda}$.
\end{theorem}

As was discussed in the introduction, both one- and two-cut cases require the same analysis in spirit. Hence, we will start with the proof of Theorem \ref{2-cut-thm} in Section \ref{two-cut-pf}, and sketch the proofs Theorems \ref{one-cut-thm}, \ref{crit-thm} in Sections \ref{one-cut-pf}, \ref{crit-pf}, respectively.

\section{Proof of Theorem \ref{2-cut-thm}}
\label{two-cut-pf}

\subsection{$g$-function}
\label{sub-g}
Before we begin our analysis of polynomials $p_{n}^{\omega}$, we will require a collection of functions and their properties, which we list here for convenience. Let
\be
g(z) := \int \log(z - s) \dd \mu_{\lambda}(s), \ z \in \C \setminus (-\infty, -1) \cup \gamma_{\lambda}
\ee
where $\log(\cdot - s)$ is holomorphic outside of $(-\infty, -1] \cup \gamma_{\lambda}[-1, s)$, where $\gamma_{\lambda}(z_1, z_2), \ z_1, z_2 \in \gamma_{\lambda}$ is the segment of $\gamma_{\lambda}$ that proceeds from $z_1$ to $z_2$. Then it follows from \eqref{q-kissing} that there is $\ell^* \in \C$ so that
\be
\label{phi-kiss-2}
g(z) = \dfrac{V(z) - \ell^*}{2} + \phi_1(z) \qandq \phi_{e}(z) := 2 \int_{e}^{z} Q_{\lambda}^{1/2}(s) \dd s, \quad e \in \{ \pm 1, z_*, -\overline{z_*} \} ,
\ee
where the domain of holomorphy for $\phi_e$ is $\C \setminus ((-\infty, -1) \cup \gamma_{\lambda})$ for $e = 1$, $\C \setminus (\gamma_{\lambda} \cup [1, \infty))$ for $e = -1$, and $\C \setminus (-\infty, -1) \cup \gamma_{\lambda}(-1, -\overline{z_*}) \cup \gamma_{\lambda}(z_*, 1) \cup [1, \infty) ) $ for $e \in \{z_*, -\overline{z_*} \}$. From Figure \ref{fig:2-cut-kiss-graph}, we immediately deduce that $ \tau \in \R$ (see \eqref{kissing-periods}) and
\be
\label{kissing-phi-1-jumps}
\phi_{1, \pm} (s) = \left\{ \begin{array}{ll}
	\pm 2 \pi \ic \mu_{\lambda}(\gamma_{\lambda}[s, 1]), & s \in \gamma_{2}, \\
	\pm 2\pi \ic \mu_{\lambda}(\gamma_{\lambda}[s, 1]) + 2\pi \ic \tau, & s \in \gamma_{1}
\end{array}\right..
\ee
Furthermore, using the fact that $\mu_{\lambda}$ is a probability measure and definition \eqref{kissing-periods} yields
\be
\label{phi-kiss-2-relations}
\phi_{1}(z) = \left\{  
\begin{array}{l}
	\phi_{z_*}(z) \pm \pi \ic   \\
	\phi_{-\overline{z_*}}(z) \pm \pi \ic  + 2\pi \ic \tau\\
	\phi_{-1}(z) \pm 2\pi \ic  + 2\pi \ic \tau
\end{array}
\right., \quad z \in \C \setminus (-\infty, -1) \cup \gamma_{\lambda} \cup (1, \infty)
\ee
and $+$ (resp. $-$) is chosen when $z$ belongs to the left (resp. right) of $(-\infty, -1) \cup \gamma_{\lambda} \cup (1, \infty)$, oriented from $-\infty$ to $\infty$, and we use the fact that 
\be
\label{omega}
\dfrac{1}{2} = -\dfrac{1}{\pi \ic} \int_{\gamma_1} Q^{1/2}_{\lambda, +}(s) \dd s
\ee
which follows from a residue calculation and the reflection symmetry of $\gamma_{1}, \gamma_2$, see \cite[Proposition 3.5]{celsus19}. With this, \eqref{kissing-phi-1-jumps}, and \eqref{phi-kiss-2} in mind, we can write 
\be
\label{g-2-cut-jump}
(g_+ - g_-)(s) = \left\{  
\begin{array}{ll}
	0, &  s \in (1, \infty),\\
	\pm \phi_{1, \pm}(s), & s \in \gamma_{2},\\
	\pi \ic , & s \in \hat{\gamma},\\
	\pm (\phi_{1, \pm}(s) - 2\pi \ic \tau), & s \in \gamma_{1},\\
	2\pi \ic, & s \in (-\infty, -1).
\end{array}
\right..
\ee
Furthermore, 
\be
\label{g-2-cut-sum}
(g_+ + g_- - V + \ell^*)(s) = \left\{  
\begin{array}{ll}
	\phi_1(s), & s \in (1, \infty),\\
	0, & s \in \gamma_2,\\
	\phi_{z_*}(s), & s \in \hat{\gamma},\\
	2 \pi \ic \tau, & s \in \gamma_1, \\
	\phi_{-1}(s) + 2\pi \ic \tau, & s \in (-\infty, -1).
\end{array}
\right..
\ee 
For $e \in \{ \pm 1 \}$, $\phi_e(z) \sim |z - e|^{1/2}$ as $z \to e$. Hence, it follows from \eqref{kissing-phi-1-jumps}, \eqref{phi-kiss-2-relations} that $(\phi_e(z))^2$ is well-defined and conformal in a small enough neighborhood of $e$, which we will denote $U_{e}$. Furthermore, it follows from \eqref{kissing-phi-1-jumps} that $(\phi_{1}(z))^2$ maps $\gamma_1\cap U_1$ into $(-\infty, 0)$ and $(\phi_{-1}(z))^2$ does the same to $\gamma_2 \cap U_{-1}$. In a similar vein, for $e \in \{ z_*, -\overline{z_*} \}, \ \phi_{e}(z) \sim |z - e|^{3/2}$ as $z \to e$. In a small neighborhood of $z = e$, \eqref{phi-kiss-2-relations} allows us to write 
\begin{equation}
\phi_{z_*, \pm}(s) = \mp 2 \pi \ic \mu_{\lambda} (\gamma_{\lambda}[z_*, s]), \quad \phi_{-\overline{z_*}, \pm}(s) = \pm 2 \pi \ic \mu_{\lambda}(\gamma_{\lambda}[s, -\overline{z_*}]).
\end{equation}
Hence, an analytic branch of $(-\phi_e)^{2/3}$ can be chosen and  $(-\phi_e)^{2/3}$ is conformal in a neighborhood of $z = e$. By the choice of $\gamma_{\lambda}$ (see Figure \ref{fig:2-cut-kiss-graph} and the second paragraph of Section \ref{riemann-surface}), both $(\phi_e(z))^2, \ e \in \{ \pm 1 \}$ and $(-\phi_e)^{2/3}, \ e \in \{z_*, -\overline{z_*} \}$ map the segments of $\gamma_1, \gamma_2$ within $U_e$ into $(-\infty, 0)$.
\subsection{Initial Riemann-Hilbert Problem}
We first deform $[-1, 1]$ to a curve $\gamma_{\lambda}$ that goes along $\gamma_{1}$, starting at $-1$ smoothly proceeds from $-\overline{z_*}$ to $z_*$ along $\hat{\gamma}$, and goes along $\gamma_2$ to $1$. To arrive at asymptotics of $p_n^{\omega}(z)$, we will use the Riemann-Hilbert approach along with Deift-Zhou nonlinear steepest descent method. The connection between the RHP below and orthogonal polynomials was first observed in the work of Fokas, Its, and Kitaev \cite{FIK91, FIK92}, while the nonlinear steepest descent method was developed by Deift and Zhou in \cite{DZ93}. More precisely, we seek a matrix $\boldsymbol Y$ that solves the following RHP (denoted \rhy)
\begin{enumerate}[(a)]
	\label{rhy}
	\item $\boldsymbol Y$ is analytic in $\C \setminus \gamma_{\lambda}$, and $\lim_{z \to \infty} \boldsymbol{Y}(z) z^{-n \sigma_3} = \boldsymbol{I}$ \footnote{Here, $\boldsymbol {I}$ is the identity and $\sigma_3 = \left(\begin{matrix}1 & 0\\ 0& -1 \end{matrix} \right).$}
	\item $\boldsymbol Y$ has continuous traces as $z \to \gamma_{\lambda}\setminus \{\pm 1 \}$ and 
	\[
	\boldsymbol{Y}_+(s) = \boldsymbol{Y}_-(s) \left(\begin{matrix} 
	1 & w_n(s) \medskip \\
	0 & 1
	\end{matrix}\right) \ \text{ for } s \in \gamma_{\lambda} \setminus \{\pm 1 \},
	\]
	where $\gamma_{\lambda}$ is oriented from $-1$ to $1$ and $w_z(z) = h(z)e^{ \ic \lambda n z}$. 
	\item As $z \to 1$, the first column of $\boldsymbol Y$ is bounded while the second behaves like $\mathcal{O}(|z - 1|^{\alpha}), \mathcal{O}(\log|z - 1|), \mathcal{O}(1)$, for $\alpha \in (-1, 0), \ \alpha = 0, \ \alpha>0$, respectively.
	Similar behavior holds as $z \to -1$ (replace $\alpha \to \beta$ and $1 \to -1$). 
\end{enumerate}
it was observed in \cite{FIK91, FIK92} that under the assumption that
\begin{equation}
	\label{assumption}
	\deg p_n^{\omega} = n \qandq \mathcal{C}(p_n^{\omega}w_n )(z) \sim z^{-(n+1)} \qasq z \to \infty, 
\end{equation}
where \( (\mathcal{C}f)(z) = \frac{1}{2\pi \ic} \int_{\gamma_{\lambda}} [f(s)/(s - z)] \dd s, \) this problem is solved by the matrix 
\begin{equation}
\label{y}
{\boldsymbol Y}(z) := \left(
\begin{matrix}
p_n^{\omega}(z) & \mathcal{C}(p_n^{\omega} w_n)(z) \medskip \\
-2\pi \ic \kappa^2_{n - 1} p_{n-1}^{\omega}(z) & -2\pi \ic \kappa^2_{n - 1}\mathcal{C}(p_{n - 1}^{\omega} w_n)(z)
\end{matrix}\right),
\end{equation}
where $\kappa_{n}$ is the leading of the orthonormal polynomials associated with $w_n(z)$, so that $\kappa_{n-1}\mathcal{C}(p_{n-1}^{\omega} w_n)(z) = z^{-n}[1 + o(1)]$ as $ \to \infty.$ Moreover, any solution of \hyperref[rhy]{\rhy} must take the form in \eqref{y} (see, for example, \cite{ApY15}).
\subsection{First Transformation}
Let $\boldsymbol T(z) := e^{n \ell^* \sigma_3} \boldsymbol Y(z) e^{-n(g(z) + \ell^*/2) \sigma_3}$. Then, $\boldsymbol T$ satisfies the following RHP, denoted \rht
\begin{enumerate}[(a)]
	\label{rht}
	\item $\boldsymbol T(z)$ is holomorphic in $\C \setminus ((-\infty, -1) \cup \gamma_{\lambda} \cup (1, \infty))$ and $\lim_{z \to \infty} \boldsymbol T = \boldsymbol I$, 
	\item $\boldsymbol T(z)$ has continuous traces on $((-\infty, -1) \cup \gamma_{\lambda} \cup (1, \infty)) \setminus \{ \pm 1, z_*, -\overline{z_*} \}$ that satisfy
	\[
	\boldsymbol T_+(s) = \boldsymbol T_-(s) \left\{
	\begin{array}{ll}
	\left( \begin{matrix}
	e^{-n (\phi_{1, +}(s) - 2\pi \ic \tau)} & h(s)e^{2n\pi \ic \tau }\\
	0 & e^{-n (\phi_{1, -} - 2\pi \ic \tau)}
	\end{matrix}\right),& s \in \gamma_1,\\
	\left( \begin{matrix}
	e^{n\pi \ic } & h(s)e^{n \phi_{z_*}(s) }\\
	0 & e^{-n\pi \ic }
	\end{matrix}\right),& s \in \hat{\gamma}\\
	\left( \begin{matrix}
	e^{-n \phi_{1, +}(s) } & h(s)\\
	0 & e^{-n \phi_{1, -}(s) }
	\end{matrix}\right),& s \in \gamma_2,
	\end{array} \right.
	\]
	\item $\boldsymbol T$ behaves the same as $\boldsymbol Y$ as $z \to \pm 1$.
\end{enumerate}

\subsection{Opening the Lenses}

\begin{figure}[ht!]
	\centering
	\includegraphics[width = 0.5 \textwidth]{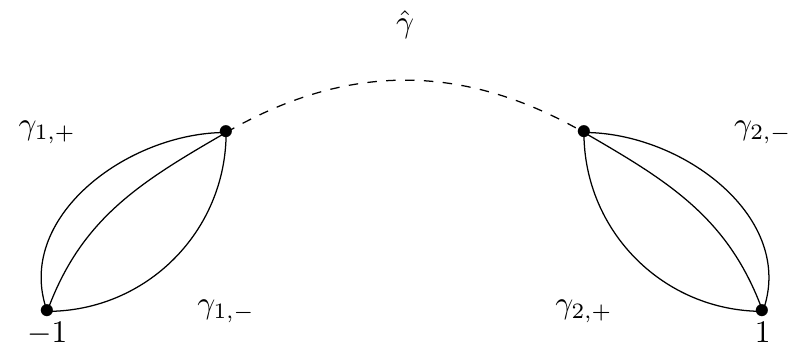}
	\caption{Opening the lenses in the supercritical regime for kissing polynomials}
	\label{fig:kiss-2-cut-lenses}
\end{figure}

Denote by $\Gamma_{i, \pm}$ the open sets delimited by $\gamma_{i, \pm}$ and $\gamma_i$. Set
\be
\label{kiss-x}
\boldsymbol X(z) := \boldsymbol T(z) \left\{
\begin{array}{ll}
	\left(\begin{matrix} 1 & 0 \\ \mp e^{-n\phi_{1}(z)}/h(z) & 1 \end{matrix}\right), & z\in \Gamma_{i_\pm}, \medskip \\
	\boldsymbol I, & \text{otherwise}.
\end{array}
\right.
\ee
Then $\boldsymbol X$ solves the following RHP (\rhx)
\begin{enumerate}[(a)]
\label{rhx}
\item $\boldsymbol X$ is analytic in $\C\setminus(\gamma_{\lambda} \cup \gamma_{i, \pm})$, $\lim_{z \to \infty} \boldsymbol X = \boldsymbol I$, 
\item $\boldsymbol X$ has continuous traces on $\gamma_{\lambda}\setminus\{\pm 1, -\overline{z_*}, z_*\}$ that satisfy \hyperref[rht]{\rht}(b) on $\hat{\gamma}$, as well as
\[
\boldsymbol X_+(s) = \boldsymbol X_-(s) \left\{
\begin{array}{rl}
e^{(2 - j)2n \pi \ic \tau \sigma_3} \left(\begin{matrix} 0 & h(s) \\ -1/h(s) & 0 \end{matrix}\right), & s\in \gamma_j, \ j = 1, 2, \medskip \\
\left(\begin{matrix}1& 0 \\ e^{-n\phi_{1}(s)}/h(s) &1\end{matrix}\right), & s\in \gamma_{i, \pm}, \ i = 1, 2,
\end{array}
\right.
\]
\item  as $z \to 1$ from outside [inside],
\[
\boldsymbol{X}(z) = \left\{  
\begin{array}{ll}
\mathcal{O} \left( \begin{matrix}
1 & |z - 1|^{\alpha}\\
1 & |z - 1|^{\alpha} 
\end{matrix} \right)  & \text{ for } -1 < \alpha < 0 \medskip \\
\mathcal{O} \left( \begin{matrix}
1 & \log|z - 1|  \\
1 & \log |z - 1|
\end{matrix}\right) & \text{ for }  \alpha = 0 \medskip \\
\mathcal{O} \left( \begin{matrix}
1 & 1  \\
1 & 1
\end{matrix}\right) & \text{ for } \alpha > 0
\end{array}
\right.
\left[ \left\{  
\begin{array}{l}
\mathcal{O} \left( \begin{matrix}
1 & |z - 1|^{\alpha}\\
1 & |z - 1|^{\alpha} 
\end{matrix} \right) \medskip \\
\mathcal{O} \left( \begin{matrix}
\log |z - 1| & \log|z - 1|  \\
\log|z - 1| & \log |z - 1|
\end{matrix}\right)  \medskip \\
\mathcal{O} \left( \begin{matrix}
|z - 1|^{\alpha} & 1  \\
|z - 1|^{\alpha} & 1
\end{matrix}\right)  
\end{array}
\right] \right.
\]
with similar behavior for $z \to -1$ where $\beta$ replaces $\alpha$.
\end{enumerate}
\subsection{Global Parametrix}
Using \eqref{kissing-phi-1-jumps}, \eqref{kissing-periods}, we see that the jumps on $\gamma_{i, \pm}$ and the off diagonal entry in the jump on $\hat{\gamma}$ are exponentially small. Hence, the RHP for the global parametrix is obtained from \hyperref[rhx]{\rhx} by removing those quantities. Thus, we are seeking a matrix $\boldsymbol N$ satisfying the following RHP (\rhn)
\begin{enumerate}[(a)]
\label{rhn}
\item $\boldsymbol N$ is analytic off of $\gamma_{\lambda}$, satisfying $\lim_{z \to \infty} \boldsymbol N = \boldsymbol I$ 
\item $\boldsymbol N$ possesses continuous traces on $\gamma_{\lambda}\setminus\{\pm 1, -\overline{z_*}, z_*\}$ that satisfy
\[
\boldsymbol N_+(s) = \boldsymbol N_-(s) \left\{
\begin{array}{ll}
	e^{(2-j)2n \pi \ic \tau \sigma_3} \left(\begin{matrix} 0 & h(s) \\ -1/h(s) & 0 \end{matrix}\right), & s\in \gamma_j, \ j = 1, 2, \medskip \\
	e^{n\pi \ic \sigma_3}, & s\in \hat{\gamma}.
\end{array}
\right.
\]
\end{enumerate}
We shall solve this problem only for \(n\in\N(\varepsilon)=\N(\lambda,\varepsilon) \) from Section~\ref{theta-kiss}.
\label{global}
To that end, let 
\begin{equation}
\label{ab7}
M_{n,0}(\z) = \Theta_{n,0}(\z) \left\{ \begin{array}{r} B(z), ~~ \z\in \RS^{(0)}, \medskip \\ A(z), ~~ \z\in \RS^{(1)}, \end{array}\right. \qandq  M_{n,1}(\z) = \Theta_{n,1}(\z) \left\{ \begin{array}{r} A(z), ~~ \z\in \RS^{(0)}, \medskip \\ -B(z), ~~ \z\in \RS^{(1)},\end{array}\right.
\end{equation}
where functions $A(z), B(z)$ are defined in \eqref{ab-kissing}. These functions are holomorphic on \( \RS\setminus\{\boldsymbol\alpha \cup \boldsymbol\beta \cup \pi^{-1} (\gamma_\lambda) \} \) since the pole of \( \Theta_{n,k}(\z) \) is canceled by the zero of \( B(z) \). Each function \( M_{n,k}(\z) \) has exactly two zeros, namely, \( \z_{n,k} \) and \( \infty^{(k)} \). It follows from \eqref{nt3} and \eqref{ab6} that
\begin{equation}
\label{Mn-jumps}
\left\{
\begin{array}{rl}
M_{n,k\pm}^{(0)}(s)=\mp M_{n,k\mp}^{(1)}(s), & s \in \gamma_{2}, \medskip \\
M_{n,k\pm}^{(0)}(s)=\mp e^{-2\pi\ic\tau n}M_{n,k\mp}^{(1)}(s), & s\in \gamma_1, \medskip \\
M_{n,k\pm}^{(i)}(s) = e^{(-1)^i \pi\ic(n + 2 c_h)}M_{n,k\mp}^{(i)}(s), & s\in\hat{\gamma}.
\end{array}
\right.\end{equation}
Then, with \( \tilde{S}_h \) as defined by \eqref{szego-kiss}, a solution of \hyperref[rhn]{\rhn} is given by
\begin{equation}
\label{rh5}
\boldsymbol N(z) = \boldsymbol M^{-1}(\infty)\boldsymbol M(z), \quad \boldsymbol M(z) := \left(\begin{matrix} M_{n,1}^{(0)}(z) & M_{n,1}^{(1)}(z) \medskip \\ M_{n,0}^{(0)}(z) & M_{n,0}^{(1)}(z) \end{matrix}\right) \tilde{S}_h^{\sigma_3}(z^{(0)}) .
\end{equation}
Indeed,  \hyperref[rhn]{\rhn}(a) follows from holomorphy of \( \tilde{S}_h(z) \) and \( M_{n,k}(\z) \) discussed in Proposition~\ref{prop:szego-kiss} and right after \eqref{ab7}. \hyperref[rhn]{\rhn}(b) can be checked by using \eqref{szego-kiss-jump} and \eqref{Mn-jumps}. It will be important for our analysis that $\boldsymbol N$ be invertible, which it is. Indeed, since the jump matrices for $\boldsymbol N$ all have determinant 1 and $\lim_{z \to \infty} \boldsymbol N(z) = \boldsymbol I$, the function $\det(\boldsymbol N(z))$ is holomorphic in $\overline{\C} \setminus \{\pm 1, -\overline{z_*}, z_*\}$, with at most square root singularities there, and hence is a constant. The normalization at infinity yields $\det(\boldsymbol N(z)) = 1$ identically.

\subsection{Local Parametrices}
\label{2-cut-local}
Let $U_{e}, \ e \in \{\pm 1 \}$ be an open disk centered at $e$ with fixed radius $\delta$ small enough so that it is in the domain of holomorphy of $h^*(z)$. We seek a matrix $\boldsymbol P_{e}, $ that solves the following \textnormal{RHP}-${\boldsymbol P_e}$:
\begin{enumerate}[(a)]
	\label{rhpe-kiss-2}
	\item[(a, b)] $\boldsymbol P_e$ satisfies \hyperref[rhx]{\rhx}(a, b, c) within $U_e$, 
	\item[(c)] $\boldsymbol P_e(z) = \boldsymbol N(z) \left( \boldsymbol I + \mathcal{O}\left( n^{-1} \right) \right)$ uniformly on $\partial U_e$ as $n \to \infty$. 
\end{enumerate}
Denote $\boldsymbol \Psi_{-1}(\zeta) := \sigma_3 \boldsymbol \Psi_{\alpha}(\zeta) \sigma_3$, $\boldsymbol \Psi_{1}(\zeta) := \boldsymbol \Psi_{\beta}(\zeta)$, where $\boldsymbol \Psi_{\alpha}$ is as in \cite[Equations (6.23) - (6.25)]{KMVaV01}. Furthermore,  $\boldsymbol \Psi_e := \sigma_3 \boldsymbol A \sigma_3$ for $e = z_*$, $\boldsymbol \Psi_e = \boldsymbol A$ for $e = -\overline{z_*}$ and $\boldsymbol A$ is the Airy matrix that appears in \cite[Section 7.6]{Deift99}. Define
\be
\label{J-kiss-2}
\boldsymbol J_e = \left\{ \begin{array}{lr}
	\boldsymbol I, & e = 1,\\
	e^{-n\pi \ic \tau},& e = -1, \\
	e^{\pm \pi\ic n\sigma_3/2}, & e=z_*, \\
	e^{\pi\ic(- \tau \pm 1/2)n\sigma_3}, & e=-\overline{z_*}.
\end{array}\right.
\ee
where the ``+" is used for $z$ to the left of $(-\infty, -1) \cup \gamma_{\lambda} \cup (1, \infty)$ and the ``-" sign is used otherwise Next, let $r_{1}(z) = \sqrt{h^{*}(z) (z + 1)^{\beta}}(z - 1)^{\alpha/2}, \ z \in U_{1}\setminus \gamma_{\lambda}$ and $(z - 1)^{\alpha/2}$ is principal, with $r_{-1}$ is defined similarly, and $r_e = \sqrt{h(z)}$ be a holomorphic branch in $U_{e}$ for $e \in \{ z_*, -\overline{z_*}\}$. Finally, let
\begin{equation}
\zeta_e(z) := \left( \dfrac{1}{4} \phi_e(z) \right)^2 , \quad e \in \{ \pm 1 \}, \quad \zeta_e(z) := \left( -\dfrac{3}{4} \phi_e(z) \right)^{2/3} , \quad e \in \{ z_{*}, -\overline{z_*} \},
\end{equation}
where $\phi_e$ is defined in \eqref{phi-kiss-2} and the branches are chosen as in Subsection \ref{sub-g}.  We now require that $\gamma_{i, \pm}$ be preimages of $I_{\pm} := \{ z \ : \ \arg(\zeta) = \pm 2\pi/3\}$. \\

It now follows by the definition of $\boldsymbol J, \boldsymbol \Psi_e, r_e$ and \eqref{phi-kiss-2}, \eqref{phi-kiss-2},\eqref{J-kiss-2}, and \eqref{phi-kiss-2-relations} that 
\be
\boldsymbol P_e(z) = \boldsymbol E_e(z) \boldsymbol \Psi_e(n^2 \zeta_e(z)) r_{e}^{-\sigma_3} e^{-n\phi_e(z) \sigma_3/2} \boldsymbol J_e
\ee
satisfies \hyperref[rhpe-kiss-2]{RHP-$\boldsymbol P_e$}(a, b). The choice of $\boldsymbol E_e$ to ensure \hyperref[rhpe-kiss-2]{RHP-$\boldsymbol P_e$}(c) holds is made below. To satisfy the matching condition \hyperref[rhpe-kiss-2]{\textnormal{RHP}-${\boldsymbol P_e}$}(c), we simply need to choose 
\be
\boldsymbol E_{e}(z) := \boldsymbol N(z) \boldsymbol J^{-1}_e r_e^{\sigma_3}(z) \boldsymbol S_e^{-1}(n^2 \zeta_e(z)),
\ee 
where $\boldsymbol S_e = \sigma_3\boldsymbol S \sigma_3$ for $e = -1$ and $\boldsymbol S_e = \boldsymbol S$ for $e = 1$, and \( \displaystyle \boldsymbol S(\zeta) := \frac{\zeta^{-\sigma_3/4}}{\sqrt2}\left(\begin{matrix} 1 & \ic \\ \ic & 1 \end{matrix}\right) \) and we take the principal branch of \( \zeta^{1/4} \). Holomorphy in $U_e \setminus \{e\}$ follows from \hyperref[kiss-2-rhn]{\rhn}(b), definition of $\boldsymbol S$, while the behavior of $\boldsymbol N$ near $e \in \{ \pm 1 \}$, the behavior of $r_e$ near $e$, and the fact that $\zeta_e(z)$ possesses a simple zero at $e$ yield holomorphy in $U_e$.
\subsection{Final Riemann-Hilbert Problem}
\label{final}
We now define 
\begin{equation}
\label{r-kiss}
\boldsymbol{R}(z) := \boldsymbol{X}(z) \left\{  
\begin{array}{ll}
\boldsymbol{N}^{-1}(z), &  z\in \C \setminus \left(\cup_{e}U_{e} \cup \gamma_{\lambda} \cup \gamma_{i, \pm}\right),\\
\boldsymbol{P}_{e}^{-1}(z), &  z \in U_e \setminus (\gamma_{\lambda} \cup \gamma_{i, \pm}),\\
\end{array}
\right.
\end{equation}
where $\partial U_{e}$ are oriented clockwise. Then, $\boldsymbol{R}(z)$ is analytic in $\C \setminus (\gamma_{i, \pm} \cup \left(\cup_e\partial U_e \right))$ and 
\begin{equation}
\label{r-jump}
\boldsymbol{R}_+(s) = \boldsymbol{R}_{-}(s) \left\{ 
\begin{array}{ll}
\boldsymbol{I} + \mathcal{O}(e^{-cn}) & \text{ for } s \in (\gamma_{\lambda} \cup \gamma_{i, \pm})\setminus U_e ,\medskip \\
\boldsymbol{I} + \mathcal{O}\left(n^{-1}\right) & \text{ for } s \in \cup_{e}\partial U_{e} .
\end{array}
\right.
\end{equation}
The first equality follows from the fact that $\re(\phi_1) >0$ on $\Gamma_{\pm}$, which follows from noting that the formula $\re(2\phi_1(z)) = \re(V(z)) - \ell - U^\mu(z)$ implies $\re(\phi_1)$ is subharmonic in a neighborhood of $z \in \gamma_{\lambda}$ and applying the maximum principle, while the second equality holds by boundedness of $\boldsymbol N$ with $n$ and construction of $\boldsymbol P_{e}$ , see \hyperref[rhpe-kiss-2]{\textnormal{RHP}-$\boldsymbol P_e$}(c). It now follows from \cite[Corollary 7.108]{Deift99} that 
\begin{equation}
\label{R-asymp}
\boldsymbol{R}(z) = \boldsymbol I + \mathcal{O} \left(n^{-1}\right) \qasq n \to \infty,
\end{equation}
uniformly for $z \in \C \setminus (\gamma_{i, \pm} \cup \left( \cup_{e} \partial U_{e} \right))$. The asymptotic formula of $p_n^{\omega}(z)$ outside the lenses and away from endpoints follows by undoing the above transformations as was done in \cite{deano14}.

\section{Sketch of Proof of Theorem \ref{one-cut-thm}}
\label{one-cut-pf}

The starting point for this analysis is the same initial problem \hyperref[rhy]{\rhy}, with $\gamma_{\lambda}$ as in Theorem \ref{kissing-geo-one-cut}. We highlight only the main steps here:
\begin{enumerate}
	\item[(a)] Using the same $g$-function as in \cite{deano14} and $\phi$ as in Theorem \ref{kissing-geo-one-cut}, we make the transformation $\boldsymbol T(z) = 2^{n \sigma_3} \boldsymbol{Y}(z) e^{-n[g(z) + \log 2] \sigma_3}$. The main difference to highlight is that the jump of $\boldsymbol T$ are slightly different:
		\[
	\boldsymbol T_{+}(s) = \boldsymbol T_-(s) \left( \begin{matrix}
	e^{-n \phi_+(s)} & h(s)\\
	0 & e^{n \phi_+(s)}
	\end{matrix} \right) \quad \text{ for } s \in \gamma_{\lambda} \setminus \{ \pm 1 \}.
	\]
	\item[(b)] We `open the lenses' in a similar fashion as well
	\begin{equation}
	\boldsymbol X(z) = \left\{ 
	\begin{array}{ll}
	\boldsymbol T(z) & z \text{ outside the lens}, \\
	\boldsymbol{T}(z) \left(\begin{matrix}1 & 0 \\
	-e^{-n\phi(z)}/h(z) & 1 \end{matrix}\right) & z \text{ on the upper lens}, \medskip \\
	\boldsymbol{T}(z) \left(\begin{matrix}1 & 0 \\
	e^{-n\phi(z)}/h(z) & 1 \end{matrix}\right) & z \text{ on the lower lens},
	\end{array}
	\right.
	\end{equation}
	where the `upper' and `lower' lips refer to Figure \ref{fig:gamma}
	\begin{figure}[ht!]
		\centering
		\includegraphics[width = 0.4 \textwidth]{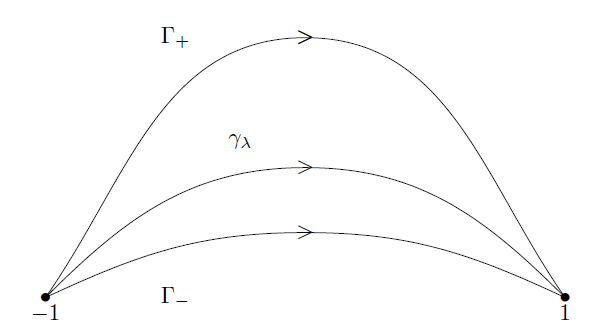}
		\caption{Curves $\Gamma_{\pm}$ and $\gamma_{\lambda}$}
		\label{fig:gamma}
	\end{figure}

	\item[(c)] To account for $h(z)$ in the weight of orthogonality, we define a different Szeg\H{o} function, which is given in \eqref{szego}.  Observe that $S_h$ is analytic and non-vanishing in $\C \setminus \gamma_{\lambda}$ and satisfies
	\begin{equation}
	\label{szego-jump}
	S_{h, +}(s) S_{h, -}(s) = (w_+ h)(s) \quad \text{ for } s \in \gamma_{\lambda} \setminus \{ \pm 1 \}.
	\end{equation} 
	Using this, we construct the global parametrix, $\boldsymbol N$ (here, $w, \varphi$ are as in \eqref{w-1})
	\begin{equation}
	\label{n}
	\boldsymbol{N}(z) := (S_h(\infty))^{\sigma_3} \left(\begin{matrix}
	1 & 1/w(z) \medskip \\
	1/2\varphi(z) & \varphi(z)/2w(z)
	\end{matrix}\right) S_h^{-\sigma_3}(z),
	\end{equation}
	\item[(d)] The local parametrices needed near $z = \pm 1$ are as in \cite{KMVaV01} to allow for a general $\alpha, \beta$ in the weight $h(z)$. Similar local analysis was done in Section \ref{2-cut-local}
	\item[(e)] The final RHP is defined in a similar fashion to what was done in Section \ref{final}
\end{enumerate}

\section{Sketch of Proof of Theorem \ref{crit-thm}}
\label{crit-pf}
In the case $\lambda = \lambda_{cr}$ curve $\gamma_{\lambda}$ seizes to be smooth, and we must modify the lenses as shown in Figure \ref{fig2}. In this setting, we will define matrices $\boldsymbol T, \boldsymbol X$, $\boldsymbol N$, and $\boldsymbol R$ in the same way as was done in the sub-critical case. However, we will need to perform some local analysis at the midpoint of $\gamma_{\lambda}$, which lies at $ 2 \ic/\lambda_{cr}$. 

\begin{figure}[ht!]
	\centering
	\includegraphics[width = 0.4 \textwidth]{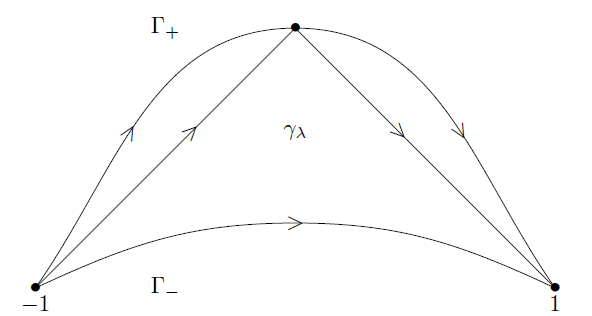}
	\caption{Curves $\Gamma_{\pm}$ and $\gamma_{\lambda}$.}
	\label{fig2}
\end{figure}

\subsection{Local Parametrix around \texorpdfstring{$2 \ic/\lambda_{cr}$}{the midpoint of gamma lambda}}

Let $U_c$ be a disk centered at $z^* = 2\ic/\lambda_{cr}$ small enough so that $h(z)$ (see the second line of Section \ref{two-cut-pf}) is holomorphic in $\overline{U}_c$, and let $\phi$ be defined as in Theorem \ref{kissing-geo-one-cut}. We seek a matrix $\boldsymbol{P}_c(z)$ to solve the following RHP (\rhpc):
\begin{enumerate}[(a)]
	\label{rhp-kissing-crit}
	\item[(a, b)] $\boldsymbol{P}_{c}(z)$ satisfies \hyperref[rhx]{\rhx}(a, b) within $U_c$ 
	\item[(c)] $\boldsymbol P_{c}(z)$ is bounded as $z \to 2\ic/\lambda_{cr}$ and $\boldsymbol{N}^{-1}(z) \boldsymbol{P}_{c}(z)  = \boldsymbol I + \mathcal{O}\left( n^{-1/2} \right)$ uniformly for $z \in \partial U_c$.
\end{enumerate}
We will need a new conformal map near the point $2\ic/\lambda_{cr}$. To this end, let $\phi_c(z) = \pm \phi(z), \ z \in U_{c, \pm}$, where $U_{c, +}$ (resp., $U_{c, -}$) is the component of $U_c$ to the left (resp., right) of $\gamma_{\lambda}$. Then, $\phi_c$ is holomorphic in $U_c$ and since $z^*$ is a simple zero of $Q_{\lambda_{cr}}^{1/2}$, we have that 
\(|\phi_c(z) - \phi_c(z^*)|  \sim |z - z^*|^2 \) as \(z \to z^*\). Furthermore, by Theorem \ref{kissing-geo-one-cut}, we have that $\phi_{\pm}(s) = \pm 2 \pi \ic \mu_{\lambda}([s, 1])$ for $s \in \gamma_{\lambda},$ and we can see that $\phi_c(z)$ is purely imaginary and positive on $\gamma_{\lambda}(-1, z^*)$ and negative purely imaginary on $\gamma_{\lambda}(z^*, 1)$. With this in mind, we can define a branch of $(\phi_c(z) - \phi(z^*))^{1/2}$ that is holomorphic and, WLOG (up to restricting $U_c$ to a smaller neighborhood) conformal in $U_c$ and maps $\gamma_{\lambda}(-1, z^*) \cap U_c$ to $\{ z \ | \ \arg(z) = \pi/4 \}, \ \gamma_{\lambda}(z^*, 1) \cap U_c$ to $\{ z \ | \ \arg(z) = 3\pi/4 \} $. Using this branch, the map 
\(
\label{zeta-crit}
\zeta_c(z) := - (\phi_c(z) - \phi_c(z^*))^{1/2}
\)
is conformal, maps $\gamma_{\lambda}(-1, z^*) \cap U_c$ into $\{ z \ | \ \arg(z) = 5\pi/4 \}$ and $\Gamma_{+}$ into $\R$. \\

Since $h(z)$ is holomorphic and nonvanishing in $U_c$, we can define a holomorphic branch of $r(z):= \sqrt{h(z)}$. Furthermore, let 
\be
\boldsymbol J(z) := \left\{ \begin{array}{ll}
	\left( \begin{matrix} 0 & -1 \\ 1 & 0  \end{matrix} \right), & z \in U_{c, +},\\
	\boldsymbol I, & z \in U_{c, -} .
\end{array}\right.
\ee
Finally, let $\boldsymbol C$ be the matrix given in \cite[Section 7.5.3]{bleher17topological} explicitly in terms of exponentials and $\text{erfc}(z)$. $\boldsymbol C$ is holomorphic in $\C \setminus \R$, satisfies the jump relation 
\[
\boldsymbol C_+(s) = \boldsymbol C_-(s) \left( \begin{matrix}
	1 & 1 \\
	0 & 1
\end{matrix}\right),
\]
and has the asymptotic expansion 
$
\label{c-asymp}
\boldsymbol C(\zeta) \sim \left(  \boldsymbol I + \sum_{k = 0}^{\infty} \left( \begin{matrix}
	0 & b_k \\ 0 & 0 
\end{matrix}\right) \zeta^{-(2k + 1)} \right) e^{- \zeta^2 \sigma_3}.
$ Let
\begin{align*}
\boldsymbol P_c(z) := \boldsymbol E_c(z) \boldsymbol C \left(\sqrt{n/2} \cdot \zeta_c(z) \right) \boldsymbol J^{-1}(z) r^{-\sigma_3} e^{-n \phi(z)\sigma_3/2}, \qquad \boldsymbol E_c(z) := \boldsymbol N(z) r^{\sigma_3}(z) \boldsymbol J(z).
\end{align*}
$\boldsymbol P_c$ satisfies \hyperref[rhp-kissing-crit]{\rhpc}(a, b) for any $\boldsymbol E_c(z)$ holomorphic in $U_c$. Furthermore, by the very definition of $\boldsymbol C, \boldsymbol J, r$, it follows that $\boldsymbol P_c$ is bounded as $z \to z^*$. 
Since the matrices involved in its definition are holomorphic in $U_c$, $\boldsymbol E_c(z)$ is holomorphic in $U_c$. \hyperref[rhp-kissing-crit]{\rhpc}(d) follows from the behavior of $\boldsymbol C(\zeta)$ as $\zeta \to \infty$ \cite[Equation (7.19)]{bleher17topological}, that $\phi_c(z^*) \in \ic \R$, and the relation 
\[
e^{-n\phi(z)\sigma_3/2} \left( \begin{matrix} 0 & -1 \\ 1 & 0  \end{matrix} \right) = \left( \begin{matrix} 0 & -1 \\ 1 & 0  \end{matrix} \right) e^{n\phi(z)\sigma_3/2}.
\] 
%
\section*{Acknowledgments }
The author is grateful to Maxim Yattselev for his guidance and the many useful discussions, suggestions, and comments. The author would also like to thank Alfredo Dea\~no and Guilherme Silva for their help, support and encouragement.


\begin{thebibliography}{99}

\bibitem{asheim12gaussian}
Asheim~A, Deano~A, Huybrechs~D, Wang~H. A Gaussian quadrature rule for oscillatory integrals on a bounded interval. arXiv preprint arXiv:1212.1293. 2012 Dec 6.

\bibitem{celsus19}
Celsus~AF, Silva~GL. Supercritical regime for the kissing polynomials. Journal of Approximation Theory. 2020 Mar 18:105408.

\bibitem{deano14}
Deaño~A. Large degree asymptotics of orthogonal polynomials with respect to an oscillatory weight on a bounded interval. Journal of Approximation Theory. 2014 Oct 1;186:33-63.

\bibitem{DHI}
Deaño~A, Huybrechs~D, Iserles~A. Computing highly oscillatory integrals. Society for Industrial and Applied Mathematics; 2017 Dec 27.

\bibitem{GoncharRakhmanov89}
Gonchar~AA, Rakhmanov~EA. Equilibrium distributions and degree of rational approximation of analytic functions. Mathematics of the USSR-Sbornik. 1989;62(2):305.


\bibitem{ApY15}
Aptekarev~AI, Yattselev~ML. Padé approximants for functions with branch points—strong asymptotics of Nuttall–Stahl polynomials. Acta Mathematica. 2015;215(2):217-80.

\bibitem{FIK91}
Fokas~AS, Its~AR, Kitaev~AV. Discrete Painlevé equations and their appearance in quantum gravity. Communications in Mathematical Physics. 1991 Dec 1;142(2):313-44.

\bibitem{FIK92}
Fokas~AS, Its~AR, Kitaev~AV. The isomonodromy approach to matric models in 2D quantum gravity. Communications in Mathematical Physics. 1992 Jul 1;147(2):395-430.

\bibitem{DZ93}
Deift~P, Zhou~X. A steepest descent method for oscillatory Riemann--Hilbert problems. Asymptotics for the MKdV equation. Annals of Mathematics. 1993 Mar 1;137(2):295-368.

\bibitem{KMVaV01}
Kuijlaars~AB, McLaughlin~K, Van Assche~W, Vanlessen~M. The Riemann-Hilbert approach to strong asymptotics for orthogonal polynomials on [-1, 1]. arXiv preprint math/0111252. 2001 Nov 23.

\bibitem{Deift99}
Deift~P. Orthogonal polynomials and random matrices: a Riemann-Hilbert approach. American Mathematical Soc.; 1999.

\bibitem{bleher17topological}
Bleher~P, Deaño~A, Yattselev~M. Topological Expansion in the Complex Cubic Log–Gas Model: One-Cut Case. Journal of Statistical Physics. 2017 Feb 1;166(3-4):784-827.


\end{thebibliography}
\end{document}